\documentclass[10pt, bezier,amstex, oneside,reqno]{amsproc}
\usepackage{amsmath, amssymb, amsthm, verbatim, euscript, amscd, textcomp, mathrsfs}
\usepackage[dvips]{graphicx}
\usepackage{epsfig}
 \usepackage[all,cmtip]{xy}

\newtheorem*{theoremM}{Main Theorem}
\newtheorem{theorem}{Theorem}[section]
\newtheorem{lemma}[theorem]{Lemma}
\newtheorem{add}[theorem]{Addendum}
\newtheorem{corollary}[theorem]{Corollary}

\theoremstyle{definition}

\newtheorem{remark}[theorem]{Remark}

\newcommand{\field}[1]{\mathbb{#1}}
\newcommand{\D}{\mathbb D}
\newcommand{\R}{\field{R}}

\newcommand{\C}{\field{C}}
\newcommand{\CC}{\mathcal{C}}
\newcommand{\LL}{\mathcal{L}}

\renewcommand{\o}{\mathfrak {o}}
\newcommand{\su}{\mathfrak {su}}
\renewcommand{\u}{\mathfrak {u}}

\newcommand{\U}{\mathcal{U}}
\newcommand{\ie}{{\it i.e., }}

\renewcommand{\epsilon}{\varepsilon}
\renewcommand{\phi}{\varphi}

\renewcommand{\H}{\mathbb H}

\author{Andrey Gogolev$^\ast$ and Federico Rodriguez Hertz$^{\ast\ast}$}
\title[Blow-ups of partially hyperbolic dynamical systems]{Surgery for partially hyperbolic dynamical systems II. \\
Blow-up of a complex curve.}
\thanks{$^\ast$The first author was partially supported by NSF grant DMS-1823150 and Simons grant 427063.\\
$^{\ast\ast}$The second author was partially supported by NSF grant DMS-1900778}

\begin{document}

\begin{abstract} 
In this paper we use the blow-up surgery introduced in~\cite{G} to produce new higher dimensional partially hyperbolic flows. The main contribution of the paper is the {\it slow-down construction} which accompanies the {\it blow-up construction}. This new ingredient allows to dispose of a rather strong {\it domination assumption} which was crucial for results in~\cite{G}. Consequently we gain more flexibility which allows to construct new {\it volume-preserving} partially hyperbolic flows as well as new examples which are not fiberwise Anosov. The latter are produced by starting with the geodesic flow on complex hyperbolic manifold which admits a totally geodesic complex curve. Then by performing the slow-down first and the blow-up second we obtain a new (volume-preserving) partially hyperbolic flows.
\end{abstract}
\maketitle

\section{Introduction}

This paper is a sequel to~\cite{G} and familiarity with~\cite{G} would help the reader. We keep the introduction brief, still we will recall all definitions which are needed.

A flow $\varphi^t\colon M \to M$ is {\it partially hyperbolic} 
if the tangent bundle $TM$ splits into $Df$-invariant continuous subbundles 
$TM= E^{s}\oplus E^c \oplus E^{u}$ such that 
\begin{equation}
\label{def_ph_flow}
 \|D\varphi^t(v^s) \| <\lambda^t< \|D\varphi^t (v^c)\| < \mu^t< \|D\varphi^t(v^u) \|,\; t\ge 1,
\end{equation}
for some Riemannian metric $\|\cdot\|$, some $\lambda<1<\mu$ and all unit vectors $v^s\in E^s$, $v^c\in E^c$ and $v^u\in E^u$. Then it is clear that the generating vector field $\dot\phi$ lies in the center subbundle $E^c$.

An invariant submanifold $N\subset M$ is called an {\it Anosov submanifold} for $\phi^t$ if $TN=E^s\oplus \dot \phi \oplus E^u$. Note that then the flow $\phi^t_N$ is given by the restriction $\varphi^t|_N$ is an Anosov flow. Further, the flow $\phi^t\colon M\to M$ is called {\it locally fiberwise } at $N$ if a neighborhood of $ N$ can be smoothly identified with $\D^k\times  N$, where $\D^k=\{x\in\mathbb R^k: \|x\|<1\}$, in such a way that the restriction $\phi^t|_{\D^k\times  N}$ has the product form
\begin{equation}
\label{def_a}
 \varphi^t(x,y)=(a^t(x), \varphi^t_{ N}(y)),
\end{equation}
where $a^t$ is a linear hyperbolic saddle flow.

\begin{remark} Note that locally fiberwise assumption in this paper is weaker than the one in~\cite{G} as we no longer require $E^s\oplus E^u$ to be tangent to the $N$-fibers in the neighborhood $\D^k\times N$. Such weakening is crucial for the examples which we consider here. Also note that the locally fiberwise assumption implies that the normal bundle to $N$ is trivial. This assumption is not crucial for our argument, but it does simplify notation and calculations a lot.
\end{remark}

Now we can blow-up $M$ along $\{0\}\times N$ by replacing each point in $\{0\}\times N$ with the projective space of lines which pass through this point perpendicularly to $N$. The blown-up manifold $\hat M$ comes with a canonical {\it blow-down map} $\pi\colon\hat M\to M$ which collapses each projective space to its base point. The preimage $\pi^{-1}(\{0\}\times N)\simeq \R P^{k-1}\times N$ is called {\it the exceptional set.}   In smooth category, $\hat M$ is the result of replacing $\D^k\times N$ with $(\D^k\#\R P^k)\times N$. We will write $\tilde\D^k$ for $(\D^k\#\R P^k)$. If the flow $\phi^t\colon M\to M$ is locally fiberwise at $N$ then it induces a flow $\hat\phi^t\colon \hat M\to\hat M$ such that the diagram

\begin{equation}
\label{diag}
\xymatrix{
\hat M\ar_\pi[d]\ar^{\hat \phi^t}[r] & \hat M\ar_\pi[d]\\
M\ar^{\phi^t}[r] & M
}
\end{equation}
commutes. The induced flow $\hat\varphi^t\colon \hat M\to\hat M$ may or may not be partially hyperbolic. 

\begin{theoremM} 
\label{thm_main}
Let $\varphi^t\colon M\to M$ be a partially hyperbolic flow with $C^1$ invariant splitting $E^s\oplus E^c\oplus E^u$ and let $N\subset M$ be an invariant Anosov submanifold of $M$. Assume that the dynamics is locally fiberwise in a neighborhood of $N$. Let $\hat\varphi^t\colon \hat M\to\hat M$ the induced flow on $\hat M$. Then there exists a partially hyperbolic flow $\tilde\varphi^t\colon\hat M\to \hat M$ which coincides with $\hat\varphi^t$ outside of a neighborhood of the exceptional set.
\end{theoremM}

The Main Theorem builds up on the earlier work~\cite{G}. However, strictly speaking, it is not a generalization of the results in~\cite{G}. Indeed in~\cite{G} the author showed that the blown-up flow $\hat\phi^t$ is itself partially hyperbolic under more restrictive assumptions, most importantly the {\it domination assumption}, which assures that the Anosov submanifold is sufficiently fast compared to the center. In this paper we have fully disposed of the domination assumption and, most interestingly, the Main Theorem applies to examples when $\hat\phi^t$ is not partially hyperbolic. The proof of the Main Theorem relies on some tools developed in~\cite{G} but also develops different technology for controlling the returns. The key basic ingredient of the proof is the {\it slow-down} construction in the neighborhood of the Anosov submanifold which provides a remedy for absence of domination. Consequently, unlike results of~\cite{G}, the construction here can only be used for flows and not for diffeomorphisms. The benefit of the slow-down construction is that we can also produce volume preserving examples which was impossible with techniques of~\cite{G}.

We proceed to describe an application of our theorem in the setting of geodesic flows on compact complex hyperbolic manifolds. Let $M$ be a compact complex hyperbolic manifold of dimension $n$ (real dimension $2n$). One can realize $M$ as a quotient space of the complex 
hyperbolic space $\H_\C^n$ by an action of a cocompact lattice in the group of biholomorphic isometries, $\Gamma\subset SU(n,1)$. Assume that there exists a compact totally geodesic complex curve $N\subset M$. Then, up to conjugating lattice $\Gamma$, the embedding $N\subset M$ is induced by the first coordinate embedding $\H_\C^1\subset \H_\C^n$. Now consider the geodesic flow on the unit tangent bundle $\phi^t\colon T^1M\to T^1M$. We view $\phi^t$ as a partially hyperbolic flow with $\dim E^s=\dim E^u=1$. Because $N$ is totally geodesic, $\phi^t$ restricts to $T^1N$. We blow-up $T^1N\subset T^1M$. It is easy to see that the induced flow $\hat\phi^t\colon\widehat{T^1M}\to \widehat{T^1M}$ is not partially hyperbolic because it has periodic orbits with dominated splittings of different dimension signatures. Further, we can check (see Section~\ref{section_example}) that all other assumptions of Main Theorem are satisfied as well. Hence we obtain the following corollary.

\begin{corollary}
\label{cor_main}
Let $M$ be a compact complex hyperbolic manifold and let $N\subset M$ be a totally geodesic complex curve. Then the blow up $\widehat{T^1M}$ of $T^1M$ along $T^1N$ supports a partially hyperbolic flow $\tilde\varphi^t\colon\hat M\to \hat M$. Moreover, the flow $\tilde\varphi^t\colon\hat M\to \hat M$ can be chosen to be an arbitrarily $C^\infty$ small perturbation of $\hat\phi^t$.
\end{corollary}

Note that the if $\phi^t$ preserves a smooth volume $m$ then $\hat\phi^t$ preserves a smooth measure $\pi^*(m)$. However the density of $\pi^*(m)$ vanishes on the exceptional set. Nevertheless, following the idea of Katok and Lewis~\cite{KL}, we adapt our Main Theorem to the conservative setting. 

\begin{add}
\label{add}
Let $N\subset M$ and $\phi^t\colon M\to M$ be as in the Main Theorem. Assume that $\phi^t$ preserves a smooth volume $m$ which has product form in the neighborhood $\D^k\times N$; that is, $m|_{\D^k\times N}=vol\otimes vol_N$, where $vol$ is the standard Euclidean volume on $\D^k$ and $vol_N$ is a smooth $\phi^t|_N$-invariant volume on $N$. Then there exists a partially hyperbolic flow on $\hat M$ which preserves a smooth non-degenerate volume.
\end{add}
The following is a non-trivial corollary.
\begin{corollary}
\label{cor_main2}
Let $M$ be a compact complex hyperbolic manifold and let $N\subset M$ be a totally geodesic complex curve. Then the blow up $\widehat{T^1M}$ of $T^1M$ along $T^1N$ supports a volume preserving partially hyperbolic flow $\tilde\varphi^t\colon\widehat{T^1M}\to \widehat{T^1M}$. 
\end{corollary}

Finally we remark that similarly to~\cite[Section 3]{G} one can take multiple blow-ups as well as connected sums along Anosov submanifolds and produce partially hyperbolic diffeomorphisms on manifolds with even more complicated topology.

\section{The proof of the Main Theorem}
\label{section2}

\subsection{Outline of the proof}
The partially hyperbolic splitting $TM=E^s\oplus E^c\oplus E^u$ for $\phi^t\colon M\to M$ induces a splitting $T\hat M=\hat E^s\oplus\hat E^c\oplus \hat E^u$ which is invariant under $D\hat\phi^t\colon T\hat M\to T\hat M$. It can be checked in local coordinates that, because the partially hyperbolic splitting is $C^1$, the induced splitting $\hat E^s\oplus\hat E^c\oplus \hat E^u$ is continuous. Under and additional domination assumption on $\varphi^t$ at $N$ (and also a stronger locally fiberwise assumption) the latter splitting is partially hyperbolic and this situation was examined in~\cite{G}. However, in general, this splitting is not partially hyperbolic. To recover partial hyperbolicity we modify $\hat\phi^t$ in the neighborhood of the exceptional set. Recall that by the locally fiberwise assumption, in the neighborhood of $N$, the generator of the flow is given by
$$
\frac{\partial\phi^t}{\partial t}(x,y)=X(x)+Y(y),
$$
where $X$ is the vector field on $\D^k$ which generates the hyperbolic saddle $a^t$ and $Y$ is the generator of $\varphi^t_N$.  We consider a smooth bump function $\rho\colon\D^k\to N$ which is radially symmetric, that is, $\rho(x)=\bar\rho(\|x\|)$ where smooth function $\bar \rho$ verifies
\begin{enumerate}
\item $\bar\rho(s)=\rho_0<1$, for $s\le\delta$;
\item $\bar\rho$ is strictly increasing on $(\delta,2\delta)$ and $|\bar\rho'(s)|<1/\delta$ for $s\in(\delta,2\delta)$;
\item $\bar\rho(s)=1$ for $s\ge 2\delta$
\end{enumerate}
Here the constant $\rho_0$ only depends on the contraction and expansion rates of $D\phi^t$ along invariant subbundles. Constant $\delta$ will need to be chosen sufficiently small. 

Given such a bump function $\rho$ we replace the flow $\phi^t|_{\D^k\times N}$ with a new flow $\phi^t_\rho$ whose generator is given by a {\it slow-down of the saddle} $X$
\begin{equation}
\label{eq_slow_down}
\frac{\partial\phi^t_\rho}{\partial t}(x,y)=\rho X(x)+Y(y)
\end{equation}
Because $\rho=1$ on the boundary of $\D^k$ the flow $\phi^t_\rho$ extends to the rest of $M$ as $\phi^t$ and then the blown-up flow $\hat\phi^t_\rho$ is the posited partially hyperbolic flow. 

Now we briefly outline the proof of partial hyperbolicity before proceeding with a more detailed argument. First note that on the $\delta$-neighborhood of $N$ the flow $\phi^t_\rho$ is a direct product of the slow saddle $a^{\rho_0t}$ and $\phi^t_N$. Therefore, by choosing $\rho_0$ small enough, the domination condition of~\cite{G} holds on the $\delta$-neighborhood and the estimates provided in~\cite{G} yield partial hyperbolicity of $\hat\phi^t_\rho$ with respect to the splitting $T\hat M=\hat E^s\oplus\hat E^c\oplus\hat E^u$ on the $\delta$-neighborhood of the exceptional set. Also, by construction, $\hat\phi^t_\rho$ coincides with $\hat\phi^t$ outside the $2\delta$-neighborhood of the exceptional set. The main technical difficulty is that the splitting $\hat E^s\oplus\hat E^c\oplus\hat E^u$ does not remain invariant as orbits cross the transition region ($\delta\le s\le2\delta$). However, one can still consider cones centered at these non-invariant distributions and verify the Cone Criterion for partial hyperbolicity. 

In what follows we will only establish {\it the splitting into unstable and center-stable subbundles. } Roughly speaking, this follows from the fact that the damage done to the cones in the transition region ($\delta\le s\le2\delta$) is controlled uniformly (in $\delta$) thanks to the second property of $\bar\rho$ and the fact that orbits spend a bounded time of order $\delta$ in the transition region.  Because all our constructions are time-symmetric, repetitions of the arguments also yields a splitting into center-unstable and stable subbundles and hence full partial hyperbolicity. 

\subsection{Cones near the exceptional set}

We will need to introduce more notation in order to proceed with the precise description of the cones and the estimates. Denote by $\tilde \D^k_{<\delta}\times N$ the $\delta$-neighborhood of the exceptional set, that is, the preimage
$$
\pi^{-1}(\{x\in\D^k:\|x\|<\delta \}\times N)
$$
Denote by $TN= E^s_N\oplus E^c_N\oplus E^u_N$ the Anosov splitting of the restriction  $\phi^t_N$ ({\it i.e.,} $E^c_N=\dot\phi^t_N$) and by  $( E^s_N\oplus E^c_N\oplus  E^u_N)\oplus H$ the product splitting on $\tilde \D^k_{<\delta}\times N$. Given a small number $\omega>0$ define the cones on $\tilde \D^k_{<\delta}\times N$
\begin{equation}
\label{eq_cones}
\begin{split}
\CC^u_\omega(x,y)=\{v\in T_{(x,y)}(\tilde \D^k_{<\delta}\times N): \measuredangle(v, E^u_N)<\omega\}\\
\CC^{cs}_\omega(x,y)=\{v\in T_{(x,y)}(\tilde \D^k_{<\delta}\times N): \measuredangle(v, E^s_N\oplus  E^c_N\oplus H)<\omega\}
\end{split}
\end{equation}

\begin{remark} The splitting $E^s_N\oplus (E^c_N\oplus H)\oplus E^u_N$ coincides with the splitting $\hat E^s\oplus\hat E^c\oplus \hat E^u$ on the exceptional set only.
\end{remark}

Recall that $\lambda<1<\mu$ are the constants from the definition of partial hyperbolicity~(\ref{def_ph_flow}). Also let $\lambda'\in(\lambda, 1]$ and $\mu'\in [1,\mu)$ be the some constants for which we have
\begin{equation*}
c^{-1}(\lambda')^t\le \|Da^t (v)\|/\|v\| \le c(\mu')^t,
\end{equation*}
where $c>0$.\footnote{Constant $\mu'$ and $\lambda'$ can be chosen to be arbitrarily close to the ``outer" and ``inner" spectral radii of $a^t$ by choosing large $c>0$.} Here $a^t$ is the hyperbolic saddle given by the locally fiberwise structure~(\ref{def_a}) and $v\in T\D^k$. Now we pick a constant $\rho_0>0$ which enters the definition of the function $\rho$ in the previous subsection such that we have the following inequality 
\begin{equation}
\label{eq_domination}
\left(\frac{\lambda'}{\mu'}\right)^{\rho_0}>\max(\lambda, \mu^{-1})
\end{equation}
This is the {\it domination condition}~\cite[(2.3)]{G} on the flow $\phi_\rho^t$. This condition yields required estimates on the cones on $\tilde \D^k_{<\delta}\times N$ for the blown-up flow. (In this paper we will focus on the case when $\rho_0<1$ because otherwise, if domination condition holds with $\rho_0=1$, our Main Theorem was already established in~\cite{G}.)  Precisely, we have the following lemma.

\begin{lemma}
\label{lemma1}
There exist $\omega>0$, $c>0$, $\kappa>1$ and $\delta_0>0$ such that for all $\delta<\delta_0$ there exists a Riemannian metric $\|\cdot\|_\delta$ on $\hat M$, which coincides with the metric $\|\cdot\|$ coming from $M$ outside the $\delta$-neighborhood of the exceptional set, such that the cone fields $\CC^u_\omega$ and $\CC^{cs}_\omega$ defined above are eventually (forward and backward) invariant under $D\phi^t_\rho$ and verify the following hyperbolic properties:
\begin{enumerate}
\item for all finite orbits segments $\{\phi^s_\rho(x,y), 0\le s\le t\}$, which are entirely contained in the $\delta$-neighborhood of the exceptional set and 
for all $ v\in \CC^u_\omega(x,y)$
$$
 \|D\phi^t_\rho(v)\|_\delta>\mu^t\|v\|_\delta,\,\, t\ge 0
$$
\item for
all finite orbits segments $\{\phi^s_\rho(x,y), 0\le s\le t\}$, which are entirely contained in the $\delta$-neighborhood of the exceptional set, for all  $v\in \CC^u_\omega(x,y)$  and for all $w\in \CC^{cs}_\omega(x,y)$ with $D\phi^tw\in \CC^{cs}_\omega(\phi^t(x,y))$
$$
\frac{\|D\phi^t_\rho(v)\|_\delta}{\|v\|_\delta}>c\kappa^t\frac{\|D\phi^t_\rho(w)\|_\delta}{\|w\|_\delta}, \,\, t\ge 0
$$
\end{enumerate}
\end{lemma}

The proof of this lemma is the basic technical ingredient of the prequel paper~\cite{G}. More precisely, the construction of appropriate Riemannian metric $\|\cdot\|_\delta$ is given in Section~5.1 of~\cite{G}. (For this construction we need to assume that the Riemannian metric $\|\cdot\|$ from the definition of partial hyperbolicity~(\ref{def_ph_flow}) on $\D^k\times N$ is a direct sum of the canonical flat metric and a metric on $N$. It was explained in Section 5.3.2 of~\cite{G} that such assumption can be made without loss of generality.) Then Lemma~5.1 of~\cite{G} gives partial hyperbolicity of the splitting $E^s_N\oplus (E^c_N\oplus H)\oplus E^u_N$. Finally, the fact that the estimates hold for the vectors in the cones (with proper choice of $\omega$) is proved in Section~5.3.4 of~\cite{G}.
 
\subsection{Control along the center in the transition domain}

Consider the transition domain $A_\delta\times N$, where $A_\delta=\tilde \D^k_{<2\delta}\cap \tilde \D^k_{>\delta}$. Recall that the Riemannian metric $\|\cdot\|_\delta$ restricted to this domain is the direct sum of the flat metric $\|\cdot\|$ and a metric on $N$. Also recall that the flow $\phi^t_\rho$ is generated by $\rho(x) X(x)+Y(y)$, $(x,y)\in A_\delta\times N$. It follows that, even though $\rho$ is not constant, the splitting $E^s_N\oplus E^c_N\oplus E^u_N \oplus H$ stays invariant within this domain. Note that because of the nature of the dynamics of the hyperbolic saddle (invariance under rescaling) and because $\rho\ge\rho_0$ with $\rho_0$ independent of $\delta$, there 
exists a uniform upper bound on time $T$ which an orbit can spend in $A_\delta\times N$ of the form
\begin{equation}
\label{eq_time}
T\le {C_1}{\delta},
\end{equation}
where $C_1$ is a constant which depends on $a^t$ and $\rho_0$, but does not depend on $\delta$ and $\rho$.

 We proceed to explain how to control extra distortion which occurs along the ``horizontal" distribution $H$. Hence we focus on the dynamics of reparametrized saddle flow $a^t_\rho$ generated by $\rho X$. 
The extra distortion which occurs along $H$ is due to $\rho$-driven shear, hence gradient of $\rho$ will appear.
We will perform all calculations in the canonical Euclidean coordinates on $A_\delta$.
Let $v\in T_xA_\delta$, let $v^t=Da^t_\rho v$ and let $v_0^t$ stand for the (isometric) translate of $v^t$ such that $v$ and $v_0^t$ have the same foot-point. Then
$$
v_0^t=v+t D(\rho X)v+h.o.t.
$$
Differentiating with respect to $t$ yields
$$
\frac{d\|v_0^t\|}{dt}\Big|_{t=0}=\frac{\langle D(\rho X)v,v\rangle}{\|v\|}
$$
Further
$$
\langle D(\rho X)v,v\rangle =\langle (\nabla\rho X)v,v\rangle+\rho\langle DXv,v\rangle\le (C_2\|\nabla\rho\|+C_3)\|v\|^2,
$$
 where $C_2$ and $C_3 (=\log\mu')$ are constants which depend only on $a^t$ (recall that $\rho\le 1$). We conclude that there exists a constant $C_4$ and a small $t_0$ such that for all $t<t_0$ and for all $v\in T_xA_\delta$, $x\in A_\delta$, we have
 $$
 \frac{\|Da_\rho^tv\|}{\|v\|}\le 1+C_4t\big(\max_{A_\delta} \|\nabla\rho\| +1\big)
 $$
Now, using this inequality we obtain the following lemma.
\begin{lemma} 
\label{lemma_transition}
Assume that an orbit segment $\{a_\rho^s(x), \, 0\le s\le T\}$ is entirely contained in $A_\delta$, then for all $v\in T_xA_\delta$, $x\in A_\delta$
$$
 \frac{\|Da_\rho^Tv\|}{\|v\|}\le C_5,\,\,\,\mbox{and}\,\,\,\, \frac{\|Da_\rho^Tv\|}{\|v\|}\ge \frac{1}{C_5},
$$
where $C_5$ is a constant which does not depend on $\rho$ and $\delta$ as long as $\delta\le 1$.
\end{lemma}
 \begin{proof}
 Pick a large $m$ such that $T/m<t_0$, then
 \begin{multline*}
  \frac{\|Da_\rho^Tv\|}{\|v\|}=\prod_{i=0}^{m-1}  \frac{\|Da_\rho^{(i+1)T/m}v\|}{\|Da_\rho^{iT/m}v\|}\le \\
  \left(1+C_4\frac Tm\big(\max_{A_\delta} \|\nabla\rho\| +1\big)\right)^m\to \exp(C_4T\big(\max_{A_\delta} \|\nabla\rho\| +1\big)), m\to\infty
 \end{multline*}
Now recall that, by the second condition on the bump function $\rho$ we have $\nabla\rho(x)=\bar\rho'(\|x\|)<1/\delta$. Using this and the upper bound on $T$ given by~(\ref{eq_time}) we obtain
$$
 \frac{\|Da_\rho^Tv\|}{\|v\|}\le\exp\left(C_4{C_1}{\delta}\left(\frac 1\delta+1\right)\right)\le C_5
$$
The second inequality of the lemma is derived by using the same argument applied to the inverse flow $a^{-t}_\rho$.
 \end{proof}
 
\subsection{Cones away from the exceptional set}
 
 To define the cones on $M\backslash(\tilde\D^k_{>2\delta}\times N)$ we use the same $\omega$ given by Lemma~\ref{lemma1} and let
 \begin{equation*}
\begin{split}
\CC^u_\omega(p)=\{v\in T_{p}(M\backslash(\tilde\D^k_{>2\delta}\times N)): \measuredangle(v,\hat E^u)<\omega\}\\
\CC^{cs}_\omega(p)=\{v\in T_{p}(M\backslash(\tilde\D^k_{>2\delta}\times N)): \measuredangle(v,\hat E^c\oplus \hat E^s)<\omega\}
\end{split}
\end{equation*}
Because $\phi^t_\rho=\phi^t$  and  $\|\cdot\|_\delta=\|\cdot\|$ on $M\backslash(\tilde\D^k_{>2\delta}\times N)$ we then have invariance and hyperbolicity properties of these cones for orbit segments which stay in $M\backslash(\tilde\D^k_{>2\delta}\times N)$ by partial hyperbolicity of the flow $\phi^t$.

\subsection{Proof of partial hyperbolicity}
To obtain partially hyperbolic splitting $E^u_\rho\oplus E^{cs}_\rho$ for $\phi^t_\rho$ we use the cone criterion applied to $\CC^u_\omega$ and $\CC^{cs}_\omega$. We recall that on $\tilde\D^k_{<\delta}\times N$ the cone families are centered at $E^u_N$ and $E^s_N\oplus E^c_N\oplus H$ while on $M\backslash (\tilde\D^k_{>2\delta}\times N)$ the cone families are centered at $\hat E^u$ and $\hat E^c\oplus \hat E^s$.
Note also that our cone families are not defined in the transition domain $A_\delta\times N$. However, we don't need to extend cones there because orbits spend a uniformly bounded time in $A_\delta\times N$.

By preceding discussion the cones are eventually invariant and and possess hyperbolic properties required by the Cone Criterion as long as the orbit stays disjoint with $A_\delta\times N$. Hence we are left to analyze the case when $\phi^s(p)\in A_\delta\times N$, $0<s<T$, with $p$ and $\phi^T(p)$ in the boundary of $A_\delta\times N$. For the sake of concreteness we can focus on the case when $p\in \partial(\tilde\D^k_{<\delta}\times N)$ and $\phi^T(p)\in \partial(\tilde\D^k_{>2\delta}\times N)$. (The other two cases $p\in \partial(\tilde\D^k_{>2\delta}\times N)$, $\phi^T(p)\in \partial(\tilde\D^k_{<\delta}\times N)$ and $p\in \partial(\tilde\D^k_{>2\delta}\times N)$, $\phi^T(p)\in \partial(\tilde\D^k_{>2\delta}\times N)$ can be treated completely analogously.)  Recall that cone aperture $\omega$ is a fixed number given by Lemma~\ref{lemma1} and is independent of $\delta$. Also recall that $\hat E^s$, $\hat E^c$ and $\hat E^u$ are continuous distributions\footnote{Here we rely on the smoothness assumption for the partially hyperbolic splitting of $\phi^t$ in an essential way.} which coincide with $E^s_N$, $E^c_N\oplus H$ and $E^u_N$, respectively, on the exceptional set. Hence for all sufficiently small $\delta$ we have
$$
dist(E^s_N\oplus E^c_N\oplus H(q), \hat E^s\oplus \hat E^c(q))<\frac\omega{10}
$$
and
$$
dist(E^u_N(q), \hat E^u(q))<\frac\omega{10}
$$
for all $q\in \tilde \D^k_{<3\delta}\times N$. Because, locally in the neighborhood of the exceptional set, the flow $\phi^t_\rho$ preserves both splittings $E^u_N\oplus(E^c_N\oplus H)\oplus E^s_N$ and $\hat E^s\oplus\hat E^c\oplus\hat E^u$ it follows that
\begin{equation*}
\begin{split}
D\phi_\rho^T(E^u_N(p))\subset \CC^u_\omega(\phi^T_\rho(p)),\\
D\phi^{-T}_\rho(\hat E^c\oplus \hat E^s(\phi^T_\rho(p)))\subset \CC^{cs}_\omega(p)
\end{split}
\end{equation*}
Combining this observation with control provided by Lemma~\ref{lemma_transition} one can easily verify the following statement.
\begin{lemma}
\label{lemma3}
There exist constants $C_6>0$ and $C_7>0$ such that for all sufficiently small $\delta>0$ and for all points $\{p, \phi^T(p)\}\subset\partial(A_\delta\times N)$ we have
\begin{equation*}
\begin{split}
D\phi_\rho^T(\CC_\omega^u(p))\subset \CC^u_{C_6\omega}(\phi^T_\rho(p)),\\
D\phi^{-T}_\rho(\CC^{cs}_\omega (\phi^T_\rho(p)))\subset \CC^{cs}_{C_6\omega}(p),\\
\|D\phi^T_\rho v\|_\delta\ge C_7\|v\|_\delta, \, v\in \CC^u_\omega(p),\\
\|D\phi^{-T}_\rho v\|_\delta\ge C_7\|v\|_\delta,\, v\in \CC^{cs}_\omega(\phi^T_\rho(p))
\end{split}
\end{equation*}
\end{lemma}

Now note that by decreasing $\delta$ we can increase the return time to the $2\delta$-neighborhood of the exceptional set, $\tilde \D^k_{<2\delta}\times N$, as much as we wish. This observation combined with Lemma~\ref{lemma3} implies that $\CC^u_\omega$ is eventually forward invariant and $\CC^{cs}_\omega$ is  eventually backward invariant for all sufficiently small $\delta$. Finally the exponential expansion of vectors in $\CC^u_\omega$ and domination of $\CC^u_\omega$ over $\CC^{cs}_\omega$ can be checked by using a standard argument: subdividing the orbit into segments and  pasting together the estimates given by Lemmas~\ref{lemma1}, \ref{lemma3} as well as hyperbolicity of cone families outside $\tilde \D^k_{<2\delta}\times N$. This arguments takes an advantage of the long return time to $\tilde \D^k_{<2\delta}\times N$ one more time. We suppress detailed estimates as they are very standard.

\section{Volume preserving modification via Katok-Lewis trick}

We first formulate a standard lemma.
\begin{lemma}
\label{lemma}
Let $M$ be a smooth manifold equipped with a smooth non-degenerate volume form $m$. Assume that a flow generated by a smooth vector field preserves $m$. Consider a smooth function $\rho\colon M\to \R$, $\rho>0$. Then the flow generated by $\rho X$ preserves $m/\rho$.
\end{lemma}
\begin{proof}  By Cartan's formula
$$
0=\LL_Xm=\iota_Xdm+d\iota_Xm=d\iota_Xm
$$
and similarly $\LL_X(m/\rho)=d\iota_X(m/\rho)$. We calculate
\begin{multline*}
\LL_{\rho X}(m/\rho)=\rho\LL_X(m/\rho)+d\rho\wedge\iota_X(m/\rho)=\rho d\iota_X(m/\rho)+\frac1\rho d\rho\wedge \iota_Xm=\\
\rho d(\frac1\rho\iota_Xm)+\frac1\rho d\rho\wedge \iota_Xm=
\rho\left(-\frac{1}{\rho^2} d\rho\wedge\iota_Xm+\frac1\rho d\iota_Xm \right)+\frac1\rho d\rho\wedge \iota_Xm=d\iota_Xm=0
\end{multline*}
\end{proof}

The goal of this section is to prove the Addendum~\ref{add}. Recall that we assume that $\phi^t\colon M\to M$ preserves a smooth volume $m$ and $m|_{\D^k\times N}
=vol\otimes vol_N$. Recall that $\phi^t_\rho$ is a slow down of $\phi^t$ along $\D^k$. By Lemma~\ref{lemma}, the flow  $\phi^t_\rho$ also locally preserves the smooth volume 
$m_\rho|_{\D^k\times N} =\frac1\rho vol\otimes vol_N$. Note that $m_\rho=m$ near the boundary and hence extend to a smooth $\phi^t_\rho$-invariant volume on the whole of $M$ which we still denote by $m_\rho$.
Because $\rho=\rho_0$ is a constant on $\D^k_{<\delta}$, we see that $m_\rho$ still have a product form $\frac{1}{\rho_0} vol\otimes vol_N$ on $\D^k_{<\delta}\times N$.

\subsection{Replacing the smooth structure}
If we equip $\D^k$ with the standard Euclidean coordinates $(x_1, x_2,\ldots , x_k)$ then
\begin{equation}
\label{eq_vol}
vol=dx_1\wedge dx_2\wedge \ldots \wedge dx_k.
\end{equation}
By commutativity of~(\ref{diag}) $\hat\phi^t_\rho$ preserves $\phi^*m_\rho$, which is a smooth measure away from the exceptional set.

Let's examine the form of $\phi^*m_\rho$ at the exceptional set. Because $\pi$ is a product, we only need to look at the pullback of $vol$ to $\tilde \D^k$ under $\tilde \D^k\to \D^k$. Recall that
$$
\tilde\D^k=\{(x_1,x_2,\ldots x_k, \ell): (x_1,x_2,\ldots x_k)\in \ell\}
$$
and that the standard smooth charts for $\tilde\D^k$ are given by extending the standard charts for the projective space $\R P^{k-1}$. Namely the $i$-th  chart is given by
\begin{multline}
\label{eq_chart}
\Psi_i(u_1, u_2, \ldots u_k)=\\
(u_1u_i, u_2u_i,\ldots u_{i-1}u_i, u_i, u_{i+1}u_i, \ldots u_ku_i, [u_1:\ldots :u_{i-1}: 1: u_{i+1}:\ldots :u_k])
\end{multline}
We can calculate the pull-back of $vol$
$$
d(u_1u_i)\wedge d(u_2u_i)\wedge\ldots\wedge du_i\wedge \ldots \wedge d(u_ku_i)=u_i^{k-1}du_1\wedge du_2\wedge\ldots\wedge du_k.
$$
Hence, when $k>1$ the pull-back vanishes on the projective space. To remedy the situation we follow the idea of Katok-Lewis (which they used to construct non-standard higher rank volume preserving group actions.) Namely we replace the smooth structure on $\D^k$ by declaring that 
$$
\Phi\colon \vec u\mapsto \|\vec u\|^\alpha \vec u, \,\alpha<0
$$
is a smooth chart near the origin (\ie by changing the smooth atlas). With respect to this chart the Euclidean norm of a vector $\vec u$ is given by
\begin{equation}
\label{eq_norm}
\|\vec u\|_\textup{new}=\|\vec u\|^{1+\alpha}
\end{equation}

Accordingly we change the smooth structure on $M$ by declaring that $\Phi\times id_N\colon\D^k\times N\to M$ is a smooth chart at $N$. Note that $M$ equipped with the new smooth atlas, which we denote by $M^{\textup{new}}$, is obviously diffeomorphic to the original $M$. However, it is easy to check that $a^t_\rho\colon\D^k\to \D^k$ and, hence, $\phi^t_\rho\colon M^{\textup{new}}\to M^{\textup{new}}$ fail to be smooth.

Accordingly we replace we replace charts~(\ref{eq_chart}) for $\tilde \D^k$ by composing $\Psi_i$ and $\Phi$, that is,
\begin{multline*}
\Psi_i^{\textup{new}}(u_1, u_2, \ldots u_k)=\\
\big(f_\alpha(u_1, \ldots, u_{i-1}, u_{i+1}, \ldots u_k)\|u_i\|^\alpha(u_1u_i, u_2u_i,\ldots u_i, \ldots u_ku_i), [u_1:\ldots :1:\ldots :u_k])\big),
\end{multline*}
where
$$
f_\alpha(u_1, \ldots, u_{i-1}, u_{i+1}, \ldots u_k)=(u_1^2+u_2^2+\ldots +u_{i-1}^2+1+u_{i+1}^2+\ldots +u_k^2)^{\alpha/2}
$$

Because the new smooth structure amounts to mere reparametrization in the radial direction the projective dynamics remains exactly the same. A direct calculation in chart shows that $\hat a_\rho^t\colon\tilde \D^k\to \tilde \D^k$ is smooth with respect to the new smooth structure. Hence $\hat\phi^t_\rho\colon \hat M^{\textup{new}}\to \hat M^{\textup{new}}$ is also smooth. Further, by appropriate choice of $\alpha$ we can now guarantee that $\pi^*m$ is a non-degenerate volume on $\hat M^{\textup{new}}$.
We present the chart calculation which determines the ``right" value of $\alpha$. In order to simplify notation we perform this calculation in the first chart $\Psi_1^{\textup{new}}$. We also abbreviate $f_\alpha(u_2, u_3, \ldots u_k)$ to simply $f_\alpha$. Note that
$$
df_\alpha\wedge du_2\wedge du_3\wedge\ldots \wedge du_k=0
$$
This is very helpful for the calculation:
\begin{align*}
 d   & (f_\alpha    \|u_1\|^\alpha u_1)\wedge d(f_\alpha\|u_1\|^\alpha u_1 u_2)\wedge\ldots \wedge d(f_\alpha\|u_1\|^\alpha u_1u_k)=\\
 &    d(f_\alpha\|u_1\|^\alpha u_1)\wedge (u_2 d(f_\alpha\|u_1\|^\alpha u_1) + f_\alpha\|u_1\|^\alpha u_1 du_2  ) \wedge \ldots \wedge (u_k d(f_\alpha\|u_1\|^\alpha u_1) + f_\alpha\|u_1\|^\alpha u_1 du_k  )=\\
&    (f_\alpha\|u_1\|^\alpha u_1)^{k-1}d(f_\alpha\|u_1\|^\alpha u_1)\wedge du_2\wedge\ldots \wedge du_k=\\
&   (f_\alpha\|u_1\|^\alpha u_1)^{k-1}\big(f_\alpha d(\|u_1\|^\alpha u_1)\wedge du_2\wedge\ldots \wedge du_k +\|u_1\|^\alpha u_1 df_\alpha\wedge du_2\wedge du_3\wedge\ldots \wedge du_k\big)=\\
&   (f_\alpha\|u_1\|^\alpha u_1)^{k-1}(\alpha+1)f_\alpha\|u_1\|^\alpha du_1\wedge du_2\wedge \ldots \wedge du_k)=(\alpha+1)f_\alpha^k\|u_1\|^{k\alpha} u_1^{k-1}
\end{align*}
Notice that $f_\alpha$ is a smooth function.
Hence the pull-back of $vol$ is a smooth and non-degenerate on $M^{\textup{new}}$ when $k\alpha+k-1=0$, \ie
$$
\alpha=-\frac{k-1}{k}
$$
\begin{remark}
\label{remark2}
It is crucial for this construction that the initial volume on $\D^k$ given by~(\ref{eq_vol}) has constant density. Indeed, if we allow for a non-trivial density $\beta(x_1,\ldots x_k)$ and begin with $\beta dx_1\wedge dx_2\wedge \ldots \wedge dx_k$ instead, then all computations go through in the same way. However the expression for the density after the blow-up in the chart $\Psi_i^{\textup{new}}$ will have an additional factor
$$
\beta(f_\alpha(u_1, \ldots, u_{i-1}, u_{i+1}, \ldots u_k)\|u_i\|^\alpha(u_1u_i, u_2u_i,\ldots u_i, \ldots u_ku_i))
$$
which is not $C^1$ at the exceptional set given by $u_i=0$ (unless the Taylor coefficients of $\beta$ up to order $k$ vanish). Hence we have a positive continuous density which is not $C^1$ on the exceptional set. This issue, in fact, gives us an additional difficulty to overcome in the proof of Corollary~\ref{cor_main2}.
\end{remark}

\subsection{Partial hyperbolicity in volume preserving setting} We now have a volume preserving flow $\hat\phi^t_\rho\colon M^{\textup{new}}\to M^{\textup{new}}$. Here we explain that this flow is also partially hyperbolic provided that constant $\rho_0$ (from the definition of $\rho$) is chosen to be sufficiently small. Namely, we amend the domination condition~(\ref{eq_domination}), as follows
\begin{equation}
\label{eq_domination2}
\left(\frac{\lambda'}{\mu'}\right)^{\rho_0}>\max(\lambda, \mu^{-1}),\,\,\, \lambda< (\lambda')^{{\rho_0}/{k}},\,\,\, (\mu')^{{\rho_0}/{k}}<\mu
\end{equation}
Clearly these inequalities are verified  for a sufficiently small $\rho_0$.

The proof of partial hyperbolicity is the same as the one given in Section~\ref{section2}. The only difference which requires some commentary is the Lemma~\ref{lemma1} for $\hat\phi^t_\rho\colon M^{\textup{new}}\to M^{\textup{new}}$ under the condition~(\ref{eq_domination2}). Recall that the proof of this lemma mostly rests on Lemma~5.1 of~\cite{G} and the proof of Lemma~5.1 is the only place which requires some adjustments. We indicate how~(\ref{eq_domination2}) must be used in the proof of Lemma~5.1. Recall that on the small neighborhood of the projective space the dynamics of $\hat a^t_\rho$ is given by
$$
\hat a^t_\rho(s,v)=(\hat{\hat a}^t_\rho(s),\bar a^t_s(v)),\,\,\, s\in\mathbb R P^{k-1}, v\in\mathbb R_+
$$
where $\hat{\hat a}^t_\rho\colon\R P^{k-1}\to\R P^{k-1}$ is the projectivization of $a^t_\rho$ (which coincides with the restriction of $\hat a^t_\rho$ to $\R P^{k-1}$) and $\bar a^t_s$ is the cocycle over $\hat{\hat a}^t_\rho$ given by the action of $a^t_\rho$ on lines (see the proof of Lemma~5.3 in~\cite{G}).\footnote{ One difference which appears is that even though, with respect to the new smooth chart $\Psi^\textup{new}$, $a^t_\rho$ still sends lines to line the cocycle $\bar a_s^t$ is no longer linear. This, however, does not present any additional difficulty.}

The estimate on $\hat{\hat a}^t_\rho$ (Claim~5.4 of~\cite{G}) remains exactly the same as the alternation of the smooth structure did not change the projective dynamics. The place where~(\ref{eq_domination2}) is needed is the inequality~(5.16) of~\cite{G} (estimate on the cocycle $\bar a_s^t$). Indeed, given a small $\vec u$, according to~(\ref{eq_norm}), we have the local estimate
$$
\|a^t_\rho(\vec u)\|_\textup{new}=\|a^t_\rho(\vec u)\|^{1+\alpha}\le(c(\mu'^\rho_0)^t\|\vec u\|)^{1+\alpha}=c^{1/k}(\mu')^{\rho_0t/k}\|\vec u\|_\textup{new}
$$
and similarly
$$
\|a^t_\rho(\vec u)\|_\textup{new}\ge c^{-1/k}(\lambda')^{\rho_0t/k}\|\vec u\|_\textup{new}
$$
This effects the last inequality in the proof of Lemma~5.3 of~\cite{G}. Namely, we obtain an exponential upper bound in 
$$
\max\left(\left(\frac{\lambda'}{\mu'}\right)^{\rho_0}, (\mu')^{\rho_0/k}\right)
$$
(and, analogously, a lower bound with $(\lambda')^{\rho_0/k}$)
Hence, in order for the rest of the proof to work we need to use~(\ref{eq_domination2}) instead of~(\ref{eq_domination}).


\section{The example}

In this section we introduce geodesic flows on complex hyperbolic manifolds in detail and then prove Corollaries~\ref{cor_main} and~\ref{cor_main2}.

\label{section_example}

\subsection{Complex hyperbolic manifolds}

First recall that 1-dimensional complex hyperbolic space can be identified with 2-dimensional real hyperbolic space with metric equal to one quarter of the standard Poincar\'e metric. The linear fractional transformations form the group of holomorphic isometries (to generate the full group of isometries one also needs the anti-holomorphic transformation) and can be identified with $PSU(1,1)=\pm Id\backslash SU(1,1)$. Because of the $\frac14$ multiple in the expression for the metric the curvature is $-4$ and the contraction and expansion rates of the geodesic flow on the complex hyperbolic space are twice bigger. It follows that the full stable and unstable horocycles of geodesic flows on higher dimensional complex hyperbolic manifolds contain one dimensional ``fast" horocycles which correspond to the complex lines in the tangent bundle. This yields a partially hyperbolic splitting which is different from the Anosov one and makes the geodesic flow on complex hyperbolic manifold suitable for the blow-up surgery.

We begin by summarizing some standard material on complex hyperbolic manifolds. We mostly follow the lucid exposition by D.B.A. Epstein~\cite{epstein}. Consider the following Hermitian quadratic forms on $\C^{n+1}$ of signature $(n,1)$.
\begin{equation*}
\begin{split}
& Q(x)=\sum_{i=1}^{n}z_i\bar z_i-z_{n+1}\bar z_{n+1}\\ 
& \hat Q(x)=\sum_{i=1}^{n-1}z_i\bar z_i+z_n\bar z_{n+1}+\bar z_n z_{n+1}
\end{split}
\end{equation*}
These forms have the following associated matrices
\begin{equation*}
\begin{split}
& J=diag(1,1,\ldots 1, -1)\\
& \hat J=\left(\begin{array}{cc}Id & 0\\ 0 & J_0\end{array}\right)
\end{split}
\end{equation*}
respectively. Here $J_0=
\left(
\begin{smallmatrix}
0 & 1\\
1 & 0
\end{smallmatrix}
\right)$.
Let $SU(n,1;Q)$ and $SU(n,1;\hat Q)$ be the groups of $(n+1)\times(n+1)$ complex matrices which have determinant 1 and preserve corresponding form. These groups are conjugate in $GL(n+1)$ by 
$$
T=\left(\begin{array}{cc}Id & 0\\ 0 & T_0\end{array}\right)
$$ where $T_0=
\frac{1}{\sqrt{2}}\left(
\begin{smallmatrix}
1 & 1\\
-1 & 1
\end{smallmatrix}
\right)$.

Recall that the complex hyperbolic $n$-space $\mathbb H^n_\C$ can be defined as
$$
\mathbb H^n_\C=\{ [x]\in \C P^n: Q(x)<0\}
$$ 
Clearly the action of $SU(n,1;Q)$ on $\C^{n+1}$ induces an action on $\mathbb H^n_\C$ and, in fact,
$SU(n,1)$ coincides with the group of biholomorphic isometries of $\mathbb H^n_\C$.
 If $\Gamma$ is a discrete cocompact subgroup of $SU(n,1)$ acting on the right then the orbit space
$$
M=\mathbb H^n_\C/\Gamma
$$
is a closed complex hyperbolic manifold. Moreover, every closed complex hyperbolic manifold arises in this way. 

\subsection{The geodesic flow as a homogenous flow}
We describe $M$ and its unit tangent bundle as homogeneous spaces. The group $SU(n,1;Q)$ acts transitively on $\mathbb H^n_\C$ and the stabilizer of $[(0,0,\ldots 0,1)]$ is 
$$
\left\{
\begin{pmatrix}
A & 0\\
0 & \overline{\det A}
\end{pmatrix}: \, A\bar A^t=Id
\right\}\simeq U(n).
$$
The stabilizer of a tangent vector is the group $W(n-1)$ given by\footnote{Notice that, by mapping to the $(n-1)\times(n-1)$ upper diagonal matrix $A$, the group $W(n-1)$ is a double cover of $U(n-1)$. It is curious to notice that, unlike in the real case, $W(n-1)$ is not isomorphic to $U(n-1)$. However using the fact that $\det\colon U(n)\to U(1)$ is a trivial principal fiber bundle one can check that $W(n-1)$ is diffeomorphic to $U(n-1)$.}
$$
W(n-1)=
\left\{
\begin{pmatrix}
A & 0 & 0\\
0 & \bar \lambda & 0\\
0 & 0 & \bar\lambda
\end{pmatrix}: \, A\bar A^t=Id, \lambda^2={\det A}
\right\}
$$
Hence we have
$$
M=U(n)\backslash SU(n,1; Q)/\Gamma\;\;\;\;\;
T^1M=W(n-1)\backslash SU(n,1; Q)/\Gamma.
$$ 

The same descriptions work using $SU(n,1;\hat Q)$ as the underlying Lie group with embeddings of $W(n-1)$ and $U(n)$ are conjugated by $T$. 
Also note that $W(0)=\{\pm Id\}$ and we will write $PSU(1,1)$ instead of $W(0)\backslash SU(1,1)$.

From now on it would be more convenient to only use the form $\hat Q$ and we abbreviate $SU(n,1;\hat Q)$ to $SU(n,1)$.

Now recall the Lie algebras
$$\u(n-1)=\{A\in M_{n-1}: \bar A^\intercal=-A\}$$
and
\begin{equation}
\label{eq_su}
\su(n,1)=\su(n,1,\hat Q)=\{B\in M_{n+1}: Tr(B)=0,\;\; \bar B^\intercal \hat J+\hat JB=0 \}
\end{equation}
If we write a traceless matrix $B\in\su(n,1)$ in block form, then $B\in\su(n,1)$ if and only if $$B=\left(\begin{array}{cc}A & v\\ -J_0\bar v^\intercal & D\end{array}\right)$$ where $A\in\o(n-1)$ and $D=
\left(
\begin{smallmatrix}
a & ib \\
ic & -\bar a
\end{smallmatrix}
\right)
$, $a\in\C$, $b, c \in \R$.

The geodesic flow $d_t\colon T^1M\to T^1M$ is given by $W(n-1)g\Gamma\mapsto d_tW(n-1)g\Gamma=W(n-1)d_tg\Gamma$, where  
$$
d_t=\left(\begin{array}{cc}Id& 0\\ 0 & d_t^0\end{array}\right),\;\;\mbox{with}\;\;\;\;d_t^0=\left(\begin{array}{cc}e^t& 0\\ 0 & e^{-t}\end{array}\right)
$$
The strong stable and strong unstable horocycle subgroups are
$$
h^{s/u}_t=\left(\begin{array}{cc}Id& 0\\ 0 & h_t^{s0/u0}\end{array}\right),\;\;\mbox{with}\;\;\;h_t^{s0}=\left(\begin{array}{cc}1& it\\ 0 & 1\end{array}\right),\;\;\;
h_t^{u0}=\left(\begin{array}{cc}1& 0\\ it & 1\end{array}\right).
$$
We refer to~\cite{FK} for a more detailed exposition on the geodesic flow as a homogeneous flow.

\subsection{Totally geodesic holomorphic curve.}

The complex hyperbolic space $\mathbb H^1_\C$ can be identified with $\{z_1=z_2=\ldots=z_{n-1}=0\}\cap \mathbb H^{n}_\C$.
The group of holomorphic isometries $SU(1,1)$ of $\mathbb H^1_\C$ embeds into $SU(n,1)$ as lower diagonal block. Let $\Gamma$ be a cocompact lattice in $SU(n,1)$ and let $\Gamma_1=SU(1,1)\cap\Gamma$. We assume that $\Gamma_1$ is a cocompact subgroup of $SU(1,1)$. Hence the embedding $\mathbb H^1_\C\subset \mathbb H^n_\C$ yields the embeddings
\begin{equation*}
\begin{split}
& N=U(1)\backslash SU(1,1)/\Gamma_1\subset U(n)\backslash SU(n,1)/\Gamma=M,\;\;\mbox{and}\\
& T^1N=PSU(1,1)/\Gamma_1\subset W(n-1)\backslash SU(n,1)/\Gamma=T^1M
\end{split}
\end{equation*}
where $N$ is a totally geodesic one dimensional complex curve.

\subsection{Parametrization of the neighborhood and the geodesic flow}

We introduce a parametrization of a neighborhood $\U$ of $PSU(1,1)$ in $W(n-1)\backslash SU(n,1)$. This parametrization will be constructed to be $\Gamma_1$ equivariant and, hence, will descend to a parametrization of a neighborhood of $T^1N$ in $T^1M$.

Pick a small $\epsilon_0>0$ and take the following as a transversal to the Lie algebras of $SU(1,1)$ and $W(n-1)$. Using the block from~(\ref{eq_su}) let
$$
\mathbb D_{\epsilon_0}=\left\{\left(\begin{array}{cc}0& v\\ -J_0\bar v^\intercal & 0\end{array}\right)\in \su(n,1),\;\;\; \mbox{where}\;\;
\|v\|<\epsilon_0\right\}
$$ 
This is a $(4n-4)$-dimensional transversal spanned by weak stable and unstable horocycles. Let $\Sigma=\Sigma_{\epsilon_0}=\exp (\mathbb D_{\epsilon_0})$.

Now we define a parametrization 
$
p\colon \Sigma\times PSU(1,1)  \to W(n-1)\backslash SU(n,1)
$
of a neighborhood $\U=\U_{\epsilon_0}$ of  $PSU(1,1)$ in $W(n-1)\backslash SU(n,1)$ as follows
\begin{equation}
\label{eq_p}
p(\sigma, u) = W(n-1)\sigma u.
\end{equation}
To verify that this is a well-defined parametrization for a sufficiently small $\epsilon_0$ it is sufficient to check that the map $P\colon W(n-1)\times\Sigma\times PSU(1,1)\to SU(n,1)$ given by $P(w,\sigma, u)=w\sigma u$ is a diffeomorphism on its image. And that the image contains a neighborhood of $W(n-1)\times PSU(1,1)\subset SU(n,1)$. To do this we consider a metric $d$ on $SU(n,1)$ which is invariant under the right action of $PSU(1,1)$ and left action of $W(n-1)$. One can obtain such a metric by starting with a right invariant Riemannian metric and then averaging with respect to the left action of (compact group) $W(n-1)$. 

Notice that $T_{id}\Sigma$, $\su(1,1)$ and ${\mathfrak w(n-1)}$ span the full Lie algebra $\su(n,1)$, and, hence, $P$ is a local diffeomorphism on the neighborhood of $(0,0,0)$. More precisely, by choosing appropriately small $\epsilon_0>0$ and $r>0$ we have that the restriction of $P$ to the neighborhood 
$$\{w\in W(n-1): d(w,id)<r\}\times \Sigma\times \{u\in PSU(1,1):  d(u, id)<r\}$$
 is a local diffeomorhism on its image. Further, because $P(w'w,\sigma, uu')=w'P(w,\sigma, u)u'$ we obtain that each point $P(w',0,u')$ has a neighborhood which has a uniform size (with respect to metric $d$) entirely contained in the image of $P$.

It remains to check that $P$ is one-to-one. Let
$$
\delta_0=\sup_{\sigma\in\Sigma}d(id,\sigma)
$$
Note that by choosing smaller $\epsilon_0$ we can make $\delta_0>0$ as small as desired. Assume that $P(w_1,\sigma_1,u_1)=P(w_2,\sigma_2, u_2)$, \ie
\begin{equation}
\label{eq_1}
w_2^{-1}w_1\sigma_1=\sigma_2u_2u_1^{-1}
\end{equation}
Then
$$
d(w_2^{-1}w_1,u_2u_1^{-1})\le d(w_2^{-1}w_1, w_2^{-1}w_1\sigma_1)+d(\sigma_2u_2u_1^{-1}, u_2u_1^{-1})= d(id, \sigma_1)+d(\sigma_2, id)\le2\delta_0
$$
Recall that $W(n-1)\times PSU(1,1)$ is (explicitly) properly embedded in $SU(n,1)$. Hence the last inequality implies that both $w_2^{-1}w_1$ and $u_2u_1^{-1}$ are close to $id$. On the other hand we have already shown that $P$ is a local diffeomorphism on the neighborhood of $id$. Hence~(\ref{eq_1}) implies that $w_2^{-1}w_1=id$, $u_2u_1^{-1}=id$ and $\sigma_1=\sigma_2$ proving that $P$ is injective.

Finally, we let $\Gamma_1$ act on $\Sigma\times PSU(1,1)$ by $\gamma_1\colon(\sigma, u)\mapsto (\sigma, u\gamma_1)$.
Our parametrization is equivariant with respect to the right action of 
$\Gamma_1$ and hence descends to a parametrization of a neighborhood of $T^1N\subset T^1M$ by $\Sigma\times PSU(1,1)/\Gamma_1\simeq \Sigma\times T^1N\simeq \mathbb D_{\epsilon_0}\times T^1N$.\footnote{Notice that in particular we have shown that the normal bundle of $T^1N$ in $T^1M$ is trivial. This happens because $W(n-1)\cap PSU(1,1)=\{Id\}$. It was pointed out to us by Mike Davis that in general the normal bundle of $N$ in $M$ is twisted and the twisting is controlled by the Chern class.}

\subsection{Proof of Corollary~\ref{cor_main}}
The Corollary~\ref{cor_main} follows from the Main Theorem provided that we verify the locally fiberwise assumption with respect to our parametrization.
We  write $v$ as a column vectors $v=(v_1,v_2)$ which parametrizes $\Sigma$. That is, 
$$
A(v_1,v_2)=\left(\begin{array}{cc}0& v\\ -J_0\bar v^\intercal & 0\end{array}\right)
$$
 and $\sigma(v_1,v_2)=\exp A(v_1,v_2)$.

Notice that 
\begin{multline*}
d_t\sigma(v_1,v_2)d_t^{-1}=d_t\exp A(v_1,v_2)d_t^{-1}\\
=\exp d_t A(v_1,v_2)d_t^{-1}=\exp A(e^{-t}v_1,e^tv_2)=\sigma(e^{-t}v_1,e^tv_2)
\end{multline*}
Now we can deduce the formula for the geodesic flow using the coordinates $(v_1,v_2,u)\in\Sigma\times PSU(1,1)$
\begin{multline*}
d_t(v_1,v_2,u)=W(n-1) d_t\sigma(v_1,v_2)u=W(n-1) d_t\sigma(v_1,v_2)d_t^{-1} d_tu\\
=(e^{-t}v_1,e^tv_2, d_tu)
\end{multline*}

We conclude that with respect to the coordinates $(v_1,v_2,u)$ the geodesic flow is the product of $(4n-4)$-dimensional hyperbolic saddle and the geodesic flow on a holomorphic curve. This verifies  the assumption of the Main Theorem on locally fiberwise structure of $d_t$ on $\U$. 

Finally to see that the partially hyperbolic flow $\tilde \phi^t$ could be chosen to be arbitrarily close to $\hat\phi^t\colon\widehat{T^1M}\to\widehat{T^1M}$ in $C^\infty$ topology recall that we obtain $\tilde \phi^t$ by blowing up the reparametrized flow $\phi^t_\rho$. The reparametrization is localized in the neighborhood of $T^1N$ and is given by~(\ref{eq_slow_down}). Function $\rho$ has to be chosen so that~(\ref{eq_domination}) holds:
\begin{equation*}
\left(\frac{\lambda'}{\mu'}\right)^{\rho_0}>\max(\lambda, \mu^{-1})
\end{equation*}
In the current setting $\lambda'^{-1}=\mu'=e$ and $\lambda^{-1}=\mu=e^2$. Hence any value of $\rho_0<1$ would work. It follows that the function $\rho$ can be chosen to be arbitrarily close to $1$ in the $C^\infty$ topology. Therefore $\phi^t_\rho$ can be arbitrarily $C^\infty$ close to $\phi^t$ and, accordingly, $\tilde\phi^t$ can be arbitrarily $C^\infty$ close to $\hat\phi^t$.

\subsection{Proof of Corollary~\ref{cor_main2}}
Corollary~\ref{cor_main2} does not immediately follow from Addendum~\ref{add}. The reason is that the pull-back of the Liouville volume form $p^*vol$ under parametrization $p$ has the form 
$$
\alpha(v_1,v_2)\omega_0\wedge vol_{T^1N},
$$
where $\omega_0$ is the standard volume form on $\D_{\epsilon_0}$ and $vol_{T^1N}$ is the Liouville volume form on $T^1N$.
Indeed, because the Liouville measure comes from the Haar measure on $SU(n,1)$ and $p$ is equivariant with respect to the right action of $PSU(1,1)$ the density $\alpha$ is independent of the $u$-coordinate. However, the dependence on $v_1$ and $v_2$ is non-trivial. Hence the Addendum~\ref{add} does not apply directly (cf. Remark~\ref{remark2}). Our approach is to replace the flow $\phi_t$ with a different flow $\bar \phi_t$ to which Addendum~\ref{add} can be applied. More precisely, on the neighborhood $\mathbb D_{\epsilon_0}\times T^1N$ we will let
\begin{equation*}
\label{new_flow}
\bar \phi^t=\bar h\circ \phi^t\circ \bar h^{-1}
\end{equation*}
where $\bar h=(h, id_{T^1N})$ and $h$ is $C^1$ small and tapers away to identity on the neighborhood of the boundary of $\mathbb D_{\epsilon_0}$.

Let $\omega_1=\alpha(v_1,v_2)\omega_0$. By rescaling $\omega_0$ if needed, we may assume that $\alpha(0,0)=1$. Denote by $a^t$ the saddle flow, $a^t(v_1,v_2)=(e^{-t}v_1, e^tv_2).$ Note that, because $\alpha$ is continuous and $a^t$-invariant, we also have $\alpha(0,v_2)=\alpha(v_1,0)=1$.
\begin{lemma}
\label{lemmata}
For all sufficiently small $\epsilon_1\in (0,\epsilon_0)$ there exists a  diffeomorphism $h\colon \D_{\epsilon_1}\to h(\D_{\epsilon_1})\subset \D_{\epsilon_0}$ such that $h_*\omega_1=\omega_0$ and $h$ commutes with the saddle flow, when defined:
$$h\circ a_t=a_t\circ h$$
\end{lemma}
Before proving the lemma we first finish the proof of Corollary~\ref{cor_main2}. First extend $h\colon \D_{\epsilon_1}\to h(\D_{\epsilon_1})$ to a diffeomorphism $h\colon \D_{\epsilon_0}\to \D_{\epsilon_0}$ which equals to identity near the boundary. Then replace the geodesic flow $\phi^t$ with the new flow $\bar \phi^t$ by replacing the restriction $\phi^t|_{ \D_{\epsilon_0}\times T^1N}$ with $(h\circ a^t\circ h^{-1}, \phi^t_{T^1N})$. Clearly $\bar\phi^t$ is smoothly conjugate to $\phi^t$. Hence $\bar\phi^t$ is partially hyperbolic with $C^1$ splitting. Further, $T^1N$ remains $\bar\phi^t$-invariant and, because $h$ commutes with $a^t$ on $\D_{\epsilon_1}$ we have 
$$
\bar \phi^t(v_1,v_2,u)=\phi^t(v_1,v_2,u)=(a^t(v_1,v_2),\phi^t_{T^1N}(u))
$$
for $(v_1,v_2)\in \D_{\epsilon_1}$. Hence the locally fiberwise assumption is also verified for $\bar \phi^t$. On the neighborhood $\D_{\epsilon_1}\times T^1N$ the $\bar\phi^t$-invariant volume has the form $\bar h_*(\omega_1\wedge vol_{T^1N})=h_*\omega_1\wedge vol_{T^1N}=\omega_0\wedge vol_{T^1N}$ and hence the assumption of Addendum~\ref{add} is also verified. We conclude that Addendum~\ref{add} applies to $\bar\phi^t$ and yields Corollary~\ref{cor_main2}. \hfill $\square$

Hence it only remains to prove the Lemma.
\begin{proof}[Proof of Lemma~\ref{lemmata}]
The idea of the proof is perform an $a^t$-equivariant Moser trick.\footnote{While such trick is standard in the context of equivariant cohomology, when the acting group is compact, see e.g.~\cite{GS}, we were unable to locate any prior work on ``locally equivariant" Moser trick. While we only do it here for the saddle singularity, presumably it is much more general.} To obtain the diffeomorphism $h$ such that $h_*\omega_1=\omega_0$ consider the path $\omega_t=(1-s)\omega_0+s\omega_1$, $s\in[0,1]$. Then, by the Poincar\'e Lemma, there exists $\eta$ such that
$$
d\eta=\omega_1-\omega_0=\gamma\omega, \,\,\,\gamma=\alpha-1
$$
Further, we can choose $\eta$ to be $a^t$-invariant; that is, $\LL_X\eta=0$, where $X=\partial a^t/\partial t$. We proceed with the proof assuming this fact, which we will verify later via a direct calculation.

Because $\omega_s$ are non-degenerate forms the equation
$$
\iota_{Y_s}\omega_s=\eta,
$$
uniquely defines ``time-dependent vector field" $Y_s$.
Then, by Cartan's formula, we have for every $s\in[0,1]$
$$
\LL_{Y_s}\omega_s=(\iota_{Y_s}\circ d+d\circ \iota_{Y_s})\omega_s=d\beta
$$
Hence by integrating $Y_s$ we obtain a one-parameter family of diffeomorophisms $h_s$ such that
$$
(h_s)_*\omega_0=\omega_s
$$
 Recall that volume forms $\omega_s$ are invariant under $X$, \ie 
$\,\,
\LL_X\omega_s=0
$
Hence
$$
0=\LL_X\beta=\LL_X(\iota_{Y_s}\omega_s)=\iota_{Y_s}(\LL_X\omega_s)+\iota_{\LL_XY_s}\omega_s=\iota_{\LL_XY_s}\omega_s,
$$
which implies that $[X, Y_s]=\LL_XY_s=0$ because $\omega_s$ in non-degenerate. It follows from the Frobenius Theorem that $a^t$ commutes with $h_s$ as posited. Note that $h_s(0,0)=(0,0)$. It remains to set $h=h_1$ and restrict to a sufficiently small disk $ \D_{\epsilon_1}$ such that $ h(\D_{\epsilon_1})\subset  \D_{\epsilon_0}$.

Hence, to finish the proof of the Lemma it remains to show that the form $\eta$ can be chosen to be $a^t$ invariant. For the sake of notation we prove this fact only when $\dim\D_{\epsilon_0}=4$. The general case can be addressed in the same way. 

We use coordinates $(x_1,x_2,x_3, x_4)$. Then $\omega_0=dx_1\wedge dx_2\wedge dx_3\wedge dx_4$ and the generator of $a^t$ is given by
$$
X=-x_1\frac{\partial}{\partial x_1}-x_2\frac{\partial}{\partial x_2}+x_3\frac{\partial}{\partial x_3}+x_4\frac{\partial}{\partial x_4}
$$
First let $\eta_0=x_1dx_2\wedge dx_3\wedge dx_4$. Then $d\eta_0=\omega_0$ and, using Cartan formula $\LL_X\eta_0=\iota_X\omega_0+d\iota_X\eta_0$ it is straightforward to verify that $\LL_X\eta_0=0$, \ie  $\eta_0$ is $a^t$-invariant.

Our goal now is to find an $a^t$-invariant function $\beta$ such that $d(\beta\eta_0)=\gamma\omega_0$. We have
$$
d(\beta\eta_0)=\beta\eta_0+d\beta\wedge\eta_0=\beta\omega_0+x_1\frac{\partial\beta}{\partial x_1}\omega
$$
Hence we need to solve the equation
$$
\beta+x_1\frac{\partial\beta}{\partial x_1}=\frac{\partial}{\partial x_1}(x_1\beta)=\gamma
$$
for $\beta$. Then
$$
\beta(x_1, x_2, x_3, x_4)=\frac{1}{x_1}\int_0^{x_1}\gamma(q,x_2,x_3,x_4)dq
$$
is a solution. 

We check that $\beta$ is $a^t$-invariant. Let $\Gamma=\Gamma(x_1,x_2,x_3,x_4)=\int_0^{x_1}\gamma(q,x_2,x_3,x_4)dq$. Because $\gamma$ is $a^t$-invariant we have
\begin{multline*}
0=\int_0^{x_1}X\gamma(q,x_2,x_3,x_4)dq=-\int_0^{x_1}q\frac{\partial}{\partial q}\gamma(q,x_2,x_3, x_4)dq-x_2 \frac{\partial}{\partial x_2} \Gamma+x_3 \frac{\partial}{\partial x_3} \Gamma+x_4 \frac{\partial}{\partial x_4} \Gamma\\
=-x_1\gamma(x_1,x_2,x_3,x_4)+\Gamma(x_1,x_2,x_3,x_4)-x_2 \frac{\partial}{\partial x_2} \Gamma+x_3 \frac{\partial}{\partial x_3} \Gamma+x_4 \frac{\partial}{\partial x_4} \Gamma
=\Gamma+X\Gamma
\end{multline*}
where we used integration by parts and the fundamental theorem of calculus. Now differentiating $x_1\beta=\Gamma$ with respect to $X$ gives
$$
X(x_1)\beta+x_1X\beta=X\Gamma
$$
which yields
$$
x_1X\beta=X\Gamma+x_1\beta=X\Gamma+\Gamma=0.
$$
Hence $X\beta=0$.
Finally by the product formula
$$
\LL_X\beta\eta_0=X(\beta)\eta_0+\beta\LL_X\eta_0=0.
$$
\end{proof}


\begin{thebibliography}{texttLL}

\bibitem[Go16]{G} Gogolev, A. {\it Surgery for partially hyperbolic dynamical systems I. Blow-ups of invariant submanifolds. } Geometry \& Topology, 22 (2018), no. 4, 2219--2252.

\bibitem[GS84]{GS} Guillemin, V.; Sternberg, S. {\it Symplectic techniques in physics. } Cambridge University Press, Cambridge, 1984. xi+468 pp.

\bibitem[E84]{epstein} Epstein, D. B. A. {\it
Complex hyperbolic geometry. Analytical and geometric aspects of hyperbolic space (Coventry/Durham, 1984)}, 93--111, 
London Math. Soc. Lecture Note Ser., 111, Cambridge Univ. Press, Cambridge, 1987. 
\bibitem[FK01]{FK} Foth, T., Katok, S., {\it
Spanning sets for automorphic forms and dynamics of the frame flow on complex hyperbolic spaces. }
Ergodic Theory Dynam. Systems 21 (2001), no. 4, 1071--1099. 
\bibitem[KL96]{KL} Katok A., Lewis J. {\it Global rigidity results for lattice actions on tori and new examples of volume-preserving actions.} Israel J. Math. 93 (1996), 253--280. 

\end{thebibliography}
\end{document}